\documentclass[a4paper,12pt]{amsart}
\usepackage{amsmath}
\usepackage{amssymb}
\usepackage{ifthen}
\usepackage{graphicx}
\usepackage{float}
\usepackage{cite}
\usepackage{amsfonts}
\usepackage{amscd}
\usepackage{amsxtra}
\usepackage{color}
\usepackage{enumerate}
\usepackage{enumitem}

\numberwithin{equation}{section}
\newtheorem{theorem}{Theorem}[section]
\newtheorem{lemma}[theorem]{Lemma}
\newtheorem{corollary}[theorem]{Corollary}

\theoremstyle{definition}
\newtheorem{remark}[theorem]{Remark}

\numberwithin{equation}{section}

 \newcounter{alphabet}
 \newcounter{tmp}
 \newenvironment{Thm}[1][]{\refstepcounter{alphabet}%
 	\bigskip%
 	\noindent%
 	{\bf Theorem \Alph{alphabet}}%
 	\ifthenelse{\equal{#1}{}}{}{ (#1)}%
 	{\bf .} \itshape}{\vskip 8pt}

 
\begin{document}


\title[Bohr-Rogosinski type inequalities for concave univalent functions]
{Bohr-Rogosinski type inequalities for concave univalent functions}

\author[Vasudeavarao Allu]{Vasudeavarao Allu}
\address{Vasudevarao Allu, School of Basic Sciences, Indian Institute of Technology Bhubaneswar,
	Bhubaneswar-752050, Odisha, India.}
\email{avrao@iitbbs.ac.in}
\author[Vibhuti Arora]{Vibhuti Arora}
\address{Vibhuti Arora, School of Basic Sciences, Indian Institute of Technology Bhubaneswar,
	Bhubaneswar-752050, Odisha, India.}
\email{vibhutiarora1991@gmail.com}

\newcommand{\D}{{\mathbb D}}
\newcommand{\C}{{\mathbb C}}
\newcommand{\real}{{\operatorname{Re}\,}}
\newcommand{\Log}{{\operatorname{Log}\,}}
\newcommand{\Arg}{{\operatorname{Arg}\,}}
\newcommand{\ds}{\displaystyle}




\begin{abstract}
In this paper, we generalize and investigate the Bohr-Rogosinski's inequalities and Bohr-Rogosinski phenomenon for the subfamilies of univalent (i.e., one-to-one) functions defined on unit disk $\mathbb{D}:=\{z\in \mathbb{C}:|z|<1 \}$ which maps to the concave domain, i.e., the domain whose complement is a convex set. All the results are proved to be sharp. 

\end{abstract}

\maketitle
\section{\bf Introduction}
Let $\mathbb{D}:=\{z\in \mathbb{C}:|z|<1 \}$ be the unit disk in $\mathbb{C}$ and $H(\mathbb{D})$ denote the family of analytic functions on $\mathbb{D}$. Let $\mathcal{S}$ be the class of all univalent functions $f\in H(\mathbb{D})$ satisfying the normalization $f(0)=0=f'(0)-1$. Let $\mathcal{B}$ be the subclass of $H(\mathbb{D})$ consisting of functions that are bounded by 1. In $1914$, Harald Bohr \cite{Bohr14} proved the following remarkable result.
 \medskip
 
 \begin{Thm}\label{ThmA}
	Let $f\in \mathcal{B}$ and $f(z)=\sum_{n=0}^{\infty}a_nz^n$
then 
\begin{equation}\label{CBI}
\sum_{n=1}^{\infty} |a_n|r^n\le 1-|f(0)|
\end{equation}
for $|z|=r\le 1/3$. The number $1/3$ cannot be improved. 
\end{Thm}
\medskip

 This constant $1/3$ and the inequality \eqref{CBI} is known as {\em Bohr radius} and {\em classical Bohr inequality}, respectively for the family $\mathcal{B}$.  Bohr originally obtained the inequality \eqref{CBI} for $r\le 1/6$, which was improved to $r\le 1/3$ by Wiener, Riesz, and Schur, independently. It is worth pointing out that if $|f(0)|$ in the classical Bohr inequality is replaced by $|f(0)|^2$, then the constant $1/3$ could be replaced by $1/2$ proved by Paulsen et al. \cite{PPS02}. 
 
 \medskip
 
 There are lots of works about the classical Bohr inequality and its generalized forms in the recent years. For example, the notion of the Bohr radius was generalized by Abu-Muhanna and Ali \cite{Abu,Abu2} to include mappings from $\mathbb{D}$ to simply connected domain and to exterior of a unit disk in $\mathbb{C}$.
 Moreover, the Bohr phenomenon for shifted disks and simply connected domains are discussed in \cite{AAH22,EPR21,FR10,KS22}.
 Allu and Halder \cite{AH21}, and Bhowmik and Das \cite{BD18} have considered the Bohr phenomenon for the class of subordinations.
 In \cite{KSS2017}, Kayumov et al. have studied the Bohr radius for locally univalent planar harmonic mappings. Kayumov and Ponnusamy have \cite{Kayumov18,Kayponn18} obtained several different improved version of the classical Bohr inequality which are sharp. Bohr-type inequalities for certain integral operators have been obtained by Kayumov et al. \cite{Kayponn19}, and Kumar and Sahoo \cite{KS2021}. 
 For the recent study on the Bohr radius, we refer to \cite{Abu4,AAL20,BD19,Kaypon18,LLP2020,LPW2020,LP2019,AKP19}
 and the references therein. The recent survey article \cite{Abu-M16} and references therein may be good sources for this topic. 
 
  \medskip

Similar to the Bohr radius, there is a concept of the Rogosinski radius \cite{R23} which is defined as follows: if $f(z)=\sum_{n=0}^{\infty} a_n z^n \in \mathcal{B}$ then $|S_M(z)|=|\sum_{n=0}^{M-1} a_n z^n|<1$ for $|z|<1/2$, where $1/2$ is the best possible quantity (see also \cite{LG86,S25}). The number $r=1/2$ is called the {\em Rogosinki Radius} for the family $\mathcal{B}$.
The {\em Bohr-Rogosinki inequality}, which is considered by Kayumov {\em et al.} in \cite{Kayponn21}, is given by
\begin{equation}\label{BRogo}
	|f(z)|+\sum_{n=N}^{\infty}|a_n|r^n\le 1=|f(0)|+d(f(0),\partial f(\mathbb{D})), 
\end{equation}
for $|z|=r\le R_N$ where $R_N$ is the positive root of the equation $2(1+r)r^N-(1-r)^2=0$. Here $d(f(0),\partial f(\mathbb{D}))$ denotes the Euclidean distance from $f(0)$ to the boundary of domain $f(\mathbb{D})$. Note that the right hand side of the inequality \eqref{BRogo} can be rewritten as
\begin{equation*}
1=1-|f(0)|+|f(0)|=|f(0)|+d(f(0),\partial f(\mathbb{D})).
\end{equation*}
To generalize the classical Bohr inequality \eqref{CBI} and Bohr-Rogosinki inequality \eqref{BRogo} to the family of function $f$ defined in $\mathbb{D}$ and takes values in any given domain $f(\mathbb{D})$, the above observation is useful.
If we replace $|f(z)|$ by $|f(0)|$ and $N=1$ in \eqref{BRogo}, then we see the connection between the Bohr-Rogosinki inequality and classical Bohr inequality. Recently, a generalization of the Bohr-Rogosinki inequality \eqref{BRogo} has been studied by Kumar and Sahoo \cite{KS21}, and Liu et al. \cite{LLP}.

\section{\bf Preliminaries}\label{Pre}
To state our main results, we need some preparation. We set
$$
\mathcal{B}_0=\{w\in \mathcal{B}:w(0)=0\}=\cup_{n=1}^{\infty}\mathcal{B}_n,
$$
where
$$
\mathcal{B}_n=\{w\in \mathcal{B}:w(0)=\cdots =w^{(n-1)}(0)=0 \text{ and } w^{(n)}(0)\ne 0\} \text{ for $n\in \mathbb{N}$}.
$$
The members of the class $\mathcal{B}_0$ are called the {\em Schwarz functions}.

\subsection{Bohr phenomenon and Bohr-Rogosinski phenomenon}

\medskip
\noindent
\vspace*{0.2cm}

For any two analytic functions $f$ and $g$ in the unit disk $\mathbb{D}$, we say that the function $g$ is {\em subordinate} to $f$, denoted by $g\prec f$ in $\mathbb{D}$, if there exist an $w\in \mathcal{B}$ with $w(0)=0$ and $g(z)=f(w(z))$ for $z\in \mathbb{D}$. Moreover, it is well known that if $f$ is univalent in $\mathbb{D}$, then $g\prec f$ if, and only if, $f(0)=g(0)$ and $f(\mathbb{D})\subset g(\mathbb{D})$. By the Schwarz lemma, it follows that
 $$
|g'(0)|=|f'(w(0))w'(0)|\le |f'(0)|.
$$

\medskip

The concept of Bohr phenomenon for the family of functions which are defined by subordination was introduced by Abu-Muhanna \cite{Abu}.
Now for a given analytic function  $f$ from $\mathbb{D}$ onto $f(\mathbb{D})$ with the expansion
\begin{equation}\label{f}
	f(z)=\sum_{n=0}^{\infty}a_nz^n,
\end{equation}
 let $S(f)=\{g:g\prec f\}$. We say that the family $S(f)$ has a {\em Bohr phenomenon} if there exists an $r_f,\,0<r_f\le 1$, such that whenever 
\begin{equation}\label{g}
	g(z)=\sum_{n=0}^{\infty}b_nz^n \in S(f),
\end{equation}
 then  
\begin{equation}\label{Bohrp}
\sum_{n=1}^{\infty}|b_n|r^n\le d(f(0),\partial f(\mathbb{D})) \mbox{ for $|z|=r\le r_f$}.
\end{equation}
 We observe that if $f(z)=(a_0-z)/(1-\overline{a_0}z)$ with $|a_0|<1$, then $f(\mathbb{D})=\mathbb{D},\,S(f)=\mathcal{B},$ and $d(f(0),\partial f(\mathbb{D}))=1-|f(0)|=1-|a_0|$ so that \eqref{Bohrp} holds with $r_f=1/3$ in view of Theorem A. For univalent functions $f$, Abu-Muhanna \cite{Abu} showed that $S(f)$ has a Bohr phenomenon and the Bohr radius is $3-2\sqrt{2}\approx 0.17157$. 

\medskip

Similar to Bohr's inequality, the Bohr-Rogosinski inequality can be generalized to the family of analytic functions $f$ in $\mathbb{D}$ which take values in a given domain $f(\mathbb{D})$. We say that the family $S(f)$ has {\em Bohr-Rogosinski phenomenon} if there exists $r_N^{f}\in(0,1]$ such that for any $g\in S(f)$ the inequality:
\begin{equation}\label{eqBR}
|g(z)|+\sum_{n=N}^{\infty}|b_n|r^n\le |f(0)|+d(f(0),\partial f(\mathbb{D}))
\end{equation}
holds for $|z|=r\le r^N_{f}$. The largest such $r^N_{f}$ is called the Bohr-Rogosinski radius.

\subsection{Bohr-Rogosinski type inequalities involving Schwarz functions}

\medskip
\noindent
\vspace*{0.2cm}

The Bohr-Rogosinski inequality \eqref{BRogo} was also generalized by replacing the Taylor coefficients $a_n$ partly or completely by higher order derivative in \cite{KS21,LiuMS}. In this sequence, Liu \cite{Liu21} has  generalized several Bohr-Rogosinski's inequalities by replacing the Taylor coefficient $a_n$ of $f$ by $f^{(n)}(w_n(z))/n!$ and $r^m$ by $|w_m(z)|$ in part or in whole, where both $w_n$ and $w_m$ are some Schwarz functions. In particular, for $N\in \mathbb{N}$, $w_i\in \mathcal{B}_{m_i},\,m_i\in \mathbb{N} \,(i=0,1,2),$ and $w^*_n\in \mathcal{B}_{h(n)}$, Liu \cite{Liu21} has proved the following Bohr-Rogosinski type inequality: 
\begin{equation}\label{eqGBR}
	|f(w_0(z))|+|f'(w_1(z))||w_2(z)|+\sum_{n=N}^{\infty}|a_n|| w^*_n(z)|\le|f(0)|+ d(f(0),\partial f(\mathbb{D})) 
\end{equation}
for certain subclasses of analytic functions.
In context of the above problem, Liu \cite{Liu21} has also generalized the notion of Bohr-Rogosinski's phenomenon \eqref{eqBR} in terms of the Schwarz functions.
Let $f,g\in H(\mathbb{D})$ be of the form \eqref{f} and \eqref{g}, respectively. Then the family $S(f)$ has a {Bohr-Rogosinski's phenomenon} in terms of the Schwarz functions if there exists an $r_{f,m_0}^{N},\,0<r_{f,m_0}^{N}\le 1$, such that whenever $g\prec f$, we have 
\begin{equation}\label{eqGBRP}
	|g(w_0(z))|+\sum_{n=N}^{\infty}|b_n|r^n\le |f(0)|+d(f(0),\partial f(\mathbb{D}))
\end{equation}
for $w_0\in \mathcal{B}_{m_0},\, m_0\in \mathbb{N}$ and $|z|=r<r_{f,m_0}^{N}$.
Note that, the Bohr-Rogosinski inequality can be deduced from the Bohr-Rogosinski type inequality by choosing proper combination of the Schwarz functions. More precisely, if we put $w_0(z)=z,\, w^*_n(z)=z^n$, and let $m_2\to \infty$ in \eqref{eqGBR}, then it reduces to \eqref{BRogo}. If we choose $w_0(z)=z$ in \eqref{eqGBRP} then it reduces to \eqref{eqBR}. To find the recent developments in this context, 
we refer to \cite{GK,Kayponn21,LLP2020}.

\medskip

One of our main concern in this article is to deal with Bohr-Rogosinski's type inequalities of the form \eqref{eqGBR} and \eqref{eqGBRP} for concave-wedge domains. For the sake of simplification, we have used the following assumptions throughout this paper:
\begin{enumerate}[label=\Roman*.]
	\item $w_i\in \mathcal{B}_{m_i}$ for $m_i\in \mathbb{N} \,(i=0,1,2)$.
	\item $w^*_n\in \mathcal{B}_{h(n)}$, for $n\in \mathbb{N}$, where $h(n)$ is some function of $n$.
\end{enumerate}

\medskip

The outline of the paper is as follows. 
In Section \ref{concave analytic}, we generalize Bohr-Rogosinki's inequality using subordination and the Schwarz functions when image domain is a concave-wedge domain.
Furthermore, we shall state our main results and some of its consequences. In the same spirit, in Section \ref{Concave pole}, we consider generalized Bohr-Rogosinki's phenomenon in terms of the Schwarz functions for the family of meromorphic univalent functions which map open unit disk $\mathbb{D}$ into some concave domain.
The section \ref{proofs} contains the proof of our main results.

\medskip

The following lemma, given by Gangania and Kumar \cite{GK}, is needed in order to prove our result.  

\medskip

\begin{lemma}\label{lemma1}
	Let $f,g\in H(\mathbb{D})$ with the Taylor expansion \eqref{f} and \eqref{g} respectively. If $g\prec f$, then 
	$$
	\sum_{n=N}^{\infty}|b_n|r^n\le \sum_{n=N}^{\infty}|a_n|r^n,\,\ N\in \mathbb{N}
	$$
	holds for $|z|= r\le 1/3$.
\end{lemma}

\medskip

The case $N=1$ of Lemma \ref{lemma1} has been proved by Bhowmik and Das \cite[Lemma 1]{BD18}.

\section{\bf The family of concave univalent function with opening angle $\pi \alpha$}\label{concave analytic}
\noindent

For our further discussions we need to introduce the following family of univalent functions:
A function $f:\mathbb{D}\to \mathbb{C}$ is said to belong to the  family of {\em concave univalent functions with opening angle $\pi \alpha,\, \alpha\in [1,2]$,} at infinity if $f$ satisfies the following conditions:
\begin{enumerate}[label=(\alph*)]
	\item $f\in H(\mathbb{D})$ is univalent and $f(1)=\infty$.
	\item $f$ maps $\mathbb{D}$ conformally onto a set whose complement with respect to $\mathbb{C}$ is convex.
	\item The opening angle of $f(\mathbb{D})$ at $\infty$ is less than or equal to $\pi \alpha,\, \alpha\in [1,2]$.
\end{enumerate}
We denote this family of functions by $\widehat{C_0}(\alpha)$. For such functions, the boundary of $f(\mathbb{D})$ is contained in a wedge shaped region with opening angle $\pi \alpha$ but not in any bigger opening angle. It may be noted that
for $f\in \widehat{C_0}(\alpha),\,\alpha\in [1,2]$, the closed set $\mathbb{C}\setminus f(\mathbb{D})$ is convex and unbounded.
Also, we observe that $\widehat{C_0}(2)$ contains the classes $ \widehat{C_0}(\alpha),\,\alpha\in [1,2]$ (see \cite{AW05}).
 When $\alpha=1$, the image domain reduces to a convex half plane. Hence we can see that concave univalent functions are related to
the convex functions and every $f\in \widehat{C_0}(1)$ is the convex function. The case $\alpha=2$ yields a slit domain. The family of normalized concave univalent function are denoted by ${C_0}(\alpha)$, i.e., ${C_0}(\alpha):=\widehat{C_0}(\alpha)\cap \mathcal{S}$. In 2005, Avkhadiev and Wirths \cite{AW05} characterized functions in the class ${C_0}(\alpha)$. For this class Fekete-Szeg\"{o} problem has been solved by Bhowmik et al. \cite{BPW11}. Results related to Yamashita conjecture on Dirichlet finite integral for ${C_0}(\alpha)$ has been discussed by Abu-Muhanna and Ponnusamy \cite{AbuP17}. For a detailed discussion about concave functions, we refer to \cite{AW05,APW06,BD18,CP07,B12} and the references therein.

\medskip

The following lemma is actually contained in the proof of \cite[Theorem 1]{BD18} (see also \cite{AW05, B12}) as well and thus we are omitting the proof.
\noindent

\medskip

\begin{lemma}\label{lemmac0}
Let $f\in \widehat{C_0}(\alpha),\,\alpha\in [1,2],$ have the expansion \eqref{f}. Then we have the following inequalities:
\begin{enumerate}[label=(\roman*)]
\item $|f'(0)|\le 2\alpha \,d(f(0),\partial f(\mathbb{D}))$.
\item $|a_n|\le A_n |f'(0)|$, for $n\ge 1$.
\item $|f(z)-f(0)|\le |f'(0)|f_{\alpha}(r)$, where $f_{\alpha}(z)$ is define by
\begin{equation}\label{Extremal}
f_{\alpha}(z):=\cfrac{1}{2\alpha}\bigg[\bigg(\cfrac{1+z}{1-z}\bigg)^{\alpha}-1\bigg]=\sum_{n=1}^{\infty}A_nz^n.
\end{equation}
\end{enumerate}
All inequalities are sharp for the function $f_{\alpha}(z)$.
\end{lemma}
\noindent

\medskip

The following is our first main result, which estimate the Bohr-Rogosinki inequality \eqref{eqGBRP} for the family $\widehat{C_0}(\alpha)$.

\medskip

\begin{theorem}\label{Thm1}
	Let $f,g\in H(\mathbb{D})$, with the Taylor expansion \eqref{f} and \eqref{g} respectively, such that $f\in\widehat{C_0}(\alpha),\,\alpha\in [1,2]$ and $g\in S(f)$. Then, for each $N\in \mathbb{N}$, the inequality
	$$
	|g(w_0(z))|+\sum_{n=  N}^{\infty} |b_n|r^n\le |f(0)|+ d(f(0),\partial f(\mathbb{D})) 
	$$
	holds for $|z|=r<\min\{r_{\alpha,m_0}^{ N},1/3\}$, where $r_{\alpha,m_0}^{ N}$ is the positive root of the equation $F_{\alpha,m_0}^{ N}(x)=0$,
\begin{equation}\label{eqcsub3}
F_{\alpha,m_0}^{ N}(x):=\sum_{n=  N}^{\infty} A_nx^n+f_{\alpha}(x^{m_0})-\cfrac{1}{2\alpha},
\end{equation}
in $(0,1)$. If $r_{\alpha,m_0}^{ N} \le 1/3$, the radius $r_{\alpha,m_0}^{ N}$ is sharp for the function $f_{\alpha}$ defined in \eqref{Extremal}. 
\end{theorem}
\noindent

\begin{remark}
Choosing $m_0\to \infty$ and $ N=1$ in Theorem \ref{Thm1}, we obtain the Bohr-Rogosinski radius
$$
r_{\alpha,\infty}^{1}=\frac{2^{1/\alpha}-1}{2^{1/\alpha}+1}
$$ and hence Theorem \ref{Thm1} coincides with \cite[Theorem 1]{BD18}. Moreover, for $\alpha=1$, it readily follows that the Bohr radius for the class $\widehat{C_0}(1)$, which is the class of convex functions, is $1/3$ and for $\alpha=2$ the Bohr radius for the class $\widehat{C_0}(2)$ is $3-2\sqrt{2}$.
\end{remark}
\noindent

\medskip

Since $\widehat{C_0}(1)$ contains in the family of univalent convex functions, the following corollary, for $\alpha=1$, covers the particular case of \cite[Corollary 1]{Liu21}.  

\medskip

\begin{corollary}
	Let $f,g\in H(\mathbb{D})$, with the Taylor expansion \eqref{f} and \eqref{g} respectively, such that $f\in\widehat{C_0}(1)$ and $g\in S(f)$. Then, for each $N\in \mathbb{N}$, the inequality
$$
|g(w_0(z))|+\sum_{n=  N}^{\infty} |b_n|r^n\le |f(0)|+ d(f(0),\partial f(\mathbb{D})) 
$$
holds for $|z|=r<\min\{r_{1,m_0}^{ N},1/3\}$, where $r_{1,m_0}^{ N}$ is the positive root of the equation
$$
F_{1,m_0}^{ N}(x)=\sum_{n=  N}^{\infty} x^n+\cfrac{r^{m_0}}{1-r^{m_0}}-\cfrac{1}{2}=0,
$$
in $(0,1)$. If $r_{1,m_0}^{ N} \le 1/3$, the radius $r_{1,m_0}^{ N}$ is sharp for the function $f_{1}$ defined in \eqref{Extremal}. 
\end{corollary}

\medskip

We observe that the case $w_0(z)=z$ in Theorem \ref{Thm1} gives the following corollary.

\medskip

\begin{corollary}
	Let $f,g\in H(\mathbb{D})$, with the Taylor expansion \eqref{f} and \eqref{g} respectively, such that $f\in\widehat{C_0}(\alpha),\,\alpha\in [1,2]$, $g\in S(f)$. Then for each $N\in \mathbb{N}$ the inequality
	$$
	|g(z)|+\sum_{n=  N}^{\infty} |b_n|r^n\le |f(0)|+ d(f(0),\partial f(\mathbb{D})) 
	$$
	holds for $|z|=r<\min\{r_{\alpha,1}^{ N},1/3\}$, where $r_{\alpha,1}^{ N}$ is the positive root of the equation $F_{\alpha,1}^{ N}=0$ is defined in \eqref{eqcsub3}. If $r_{\alpha,1}^{ N} \le 1/3$, the radius $r_{\alpha,1}^{ N}$ is sharp for the function $f_{\alpha}$ given by \eqref{Extremal}. 
\end{corollary}
\noindent

\medskip

\begin{theorem}\label{Thm2}
	Let $f\in \widehat{C_0}(\alpha),\,\alpha\in [1,2]$, with the Taylor expansion \eqref{f}. Then for each $N\in \mathbb{N}$, we have
	\begin{equation*}
		|f(w_0(z))|+|f'(w_1(z))||w_2(z)|+\sum_{n= N}^{\infty}|a_n|| w^*_n(z)|\le|f(0)|+ d(f(0),\partial f(\mathbb{D})) 
	\end{equation*}
	for $|z|=r\le r_{\alpha,m_0,m_1,m_2}^{{N}}$, where $r_{\alpha,m_0,m_1,m_2}^{{N}}\in(0,1)$ is the unique positive root of the equation $K_{\alpha,m_0,m_1,m_2}^{{N}}(x)=0$,
	\begin{equation}\label{eqC2}
		K_{\alpha,m_0,m_1,m_2}^{{N}}(x):=\sum_{n=N}^{\infty}A_nx^{h(n)}+f_{\alpha}(x^{m_0})+x^{m_2}\cfrac{(1+x^{m_1})^{\alpha-1}}{(1-x^{m_1})^{\alpha+1}}-\cfrac{1}{2\alpha}
	\end{equation}
and $A_n$ is given in \eqref{Extremal}. The radius $r_{\alpha,m_0,m_1,m_2}^{{N}}$ cannot be improved. 
\end{theorem}

\medskip

If we allow $m_2$ tends to infinity in Theorem \ref{Thm2}, then it leads to the following result.

\medskip

\begin{corollary}\label{Coro1}
	For $\alpha\in [1,2]$, let $f\in  \widehat{C_0}(\alpha)$ with the Taylor expansion \eqref{f}. Then for each $N\in \mathbb{N}$, we have
	$$
	|f(w_0(z))|+\sum_{n=N}^{\infty}|a_n|| w^*_n(z)|\le|f(0)|+ d(f(0),\partial f(\mathbb{D}))
	$$
	for $|z|=r\le r_{\alpha,m_0,m_1,\infty}^{{N}}$, where $r_{\alpha,m_0,m_1,\infty}^{{N}}\in (0,1)$ is the unique positive root of the equation $K_{\alpha,m_0,m_1,\infty}^{{N}}(x)=0$ defined in \eqref{eqC2}.
	 The radius $r_{\alpha,m_0,m_1,\infty}^{{N}}$ cannot be improved.
\end{corollary}

\medskip

Furthermore, the following result can be derived from Corollary \ref{Coro1} by letting $m_0,m_2\to \infty$.

\medskip

\begin{corollary}\label{coro2}
	Let $f\in \widehat{C_0}(\alpha),\,\alpha\in [1,2]$, with the Taylor expansion \eqref{f}. Then for each $N\in \mathbb{D}$, we have 
	\begin{equation*}
		\sum_{n=N}^{\infty}|a_n|| w^*_n(z)|\le d(f(0),\partial f(\mathbb{D})),\quad |z|=r\le r_{\alpha,\infty,m_1,\infty}^{ N},
	\end{equation*}
	where $r_{\alpha,\infty,m_1,\infty}^{ N}$ is the unique positive root of the equation $K_{\alpha,\infty,m_1,\infty}^{{N}}(x)=0$ defined in \eqref{eqC2}.
	The radius $r_{\alpha,\infty,m_1,\infty}^{ N}$ cannot be improved.
\end{corollary}


\section{\bf The family of concave univalent functions with pole $p$}\label{Concave pole}

Motivated by the work of Bhowmik and Das \cite{BD18}, we generalize Bohr-Rogosinski's phenomenon for functions that are subordinate to a meromorphic function defined in the unit disk $\mathbb{D}$. Recall from \cite{AW09}, the definition of subordination for meromorphic functions is same as for analytic functions. 

\medskip

Analogous to the family of convex analytic functions, it is interesting to consider the family
of meromorphic concave functions. Let $\widehat{\mathbb{C}}:=\mathbb{C}\cup \{\infty\}$ be the extended complex plane.
We denote by $\widehat{C_p}$ the class of meromorphic univalent functions $f:\mathbb{D}\to \widehat{\mathbb{C}}$ which satisfies the following conditions: 
\begin{enumerate}[label=(\roman*)]
	\item $f$ is analytic in $\mathbb{D}\setminus\{p\}$ and $\widehat{\mathbb{C}}\setminus f(\mathbb{D})$ is convex domain.
	\item $f$ has a simple pole at the point $p$.
\end{enumerate}  
By a suitable rotation, without loss of generality we can assume that
$0 < p < 1$. Each function in $\widehat{C_p}$ is called a {\em concave univalent function with a pole $p\in(0,1)$} and has a Taylor series expansion \eqref{f} in the disk $\mathbb{D}_p:=\{z\in \mathbb{D}:|z|<p\}$.  Let ${C_p}:=\{f\in \widehat{C_p}:f(0)=f'(0)-1=0 \}$. In 2006, Wirths \cite{KJW06} established the representation formula for functions in ${C_p}$. Using this representation, the discussion
based on the Laurent series expansion about the pole $p$ was
started by Bhowmik et al. \cite{BPW07} in 2007, where the authors obtained some coefficient estimates for the family $\widehat{C_p}$. Such
functions are intensively studied by many authors, we refer to the papers \cite{AW07conj, APW04,AW07,KJW03, BD18} for a detailed discussion about this class.

 \medskip

Using the definition of Bohr-Rogosinski's phenomenon for analytic functions, we introduce the notion of Bohr-Rogosinski's phenomenon for the family $S(f)$, where $f$ is meromorphic with pole $p$. We say that $S(f)$ has the Bohr-Rogosinski phenomenon if for any $g\in S(f)$ where $f$ and $g$ have Taylor expansion \eqref{f} and \eqref{g} in $\mathbb{D}_p$ respectively, if there exists $r_f^N,\,0<r_f^N\le p$ such that the inequality \eqref{eqBR} holds. Similar to the analytic case, this can be generalized to \eqref{eqGBR} in terms of the Schwarz function for further investigation. 
 
  \medskip
  
In view of \cite[Theorem 1]{AW07} (see also \cite[Theorem 8.4]{AW09}) for normalized functions $f\in {C_p}$, we can easily obtain the following result.

\medskip

\begin{lemma}\label{lemmacp}
Let $p\in(0,1)$. If $f\in \widehat{C_p}$ and $g=\sum_{n=0}^{\infty}b_nz^n\prec f(z)$, then 
$$
|b_n|\le |f'(0)| \cfrac{1}{p^{n-1}}\sum_{k=0}^{n-1}p^{2k}, \quad n\in \mathbb{N}.
$$
The inequality is sharp for 
\begin{equation}\label{Extremalp}
k_p(z)=\cfrac{pz}{(p-z)(1-pz)}=\sum_{n=1}^{\infty}\cfrac{1-p^{2n}}{(1-p^2)p^{n-1}}z^n=\sum_{n=1}^{\infty}c_n(p)z^n.
\end{equation}
\end{lemma}

\begin{proof}
For $f\in \widehat{C_p}$, the function defined by $F(z)=(f(z)-f(0))/f'(0)$ belongs to the class ${C_p}$. Note that $f(0)=g(0)$ and $f\prec g$ if, and only if,
$$
\cfrac{1}{f'(0)}\sum_{n=1}^{\infty}b_n z^n=\cfrac{g(z)-g(0)}{f'(0)}\prec \cfrac{f(z)-f(0)}{f'(0)}=F(z), \quad z\in \mathbb{D}.
$$
Therefore, by \cite[Theorem 1]{AW07}, we easily get
$$
\cfrac{b_n}{f'(0)}\le \cfrac{1}{p^{n-1}}\sum_{k=0}^{n-1}p^{2k}.
$$ 
\end{proof}
\noindent

We now state our next result, which deals with the Bohr-Rogosinki phenomenon for the family $\widehat{C_p}$.

\medskip
\begin{theorem}\label{Thm4}
	If $f\in \widehat{C_p},\,p\in(0,1),$ and $g(z)=\sum_{n=0}^{\infty}b_nz^n\in S(f)$. Then for each $N\in \mathbb{N}$ the inequality
\begin{equation}\label{eqcpbohr}
|g(w_0(z))|+\sum_{n=N}^{\infty}|b_n|z^n\le |f(0)|+d(f(0),\partial f(\mathbb{D}))
\end{equation}
holds for $|z|=r\le r_{p,m_0}^{ N}<p$, where $r_{p,m_0}^{ N}$ is the positive root of the equation $G_{p,m_0}^{ N}(x)=0$,
\begin{equation}\label{eqcp}
G_{p,m_0}^{ N}(x):=\sum_{n=  N}^{\infty}\cfrac{1-p^{2n}}{(1-p^2)p^{n-1}}\,x^n+k_p(x^{m_0})-\cfrac{p}{(1+p)^2}.
\end{equation}
The radius is sharp for the function $f=k_p$ given by \eqref{Extremalp}.
\end{theorem}
\noindent

\begin{remark}
We first observe that for $ N=1$ and $m_0\to \infty$, we have 
$$
r_{p,m_0}^{ N}=(p+1/p+1)-(\sqrt{p}+1/\sqrt{p})\sqrt{p+1/p}
$$ 
as the root of the equation $pr^2-2(p^2+1+p)r+p=0$ and thus Theorem \ref{Thm4} contains the result of \cite[Corollary 1]{BD18} as a special case.
\end{remark}

\medskip

Furthermore, the substitution $w_0(z)=z$ bring Theorem \ref{Thm4} back into the following form.

\medskip

\begin{corollary}
	If $f\in \widehat{C_p},\,p\in(0,1),$ with the Taylor series expansion \eqref{f} and $g(z)=\sum_{n=0}^{\infty}b_nz^n\in S(f)$. Then for each $N\in \mathbb{N}$, the inequality
	\begin{equation*}
		|g(z)|+\sum_{n=N}^{\infty}|b_n|z^n\le |f(0)|+d(f(0),\partial f(\mathbb{D}))
	\end{equation*}
	holds for $|z|=r\le r_{p,1}^{ N}<p$, where $r_{p,1}^{ N}$  is the positive root of the equation $G_{p,1}^{ N}(r)=0$ given by \eqref{eqcp}. The radius is sharp for the function $f(z)=k_p(z)$ as defined in \eqref{Extremalp}.
\end{corollary}

\section{\bf Proof of the main results}\label{proofs}
In order to prove our main results, we frequently use a consequence of the Schwarz lemma which says that ``{\em if $w\in \mathcal{B}_m$ for some $m\in \mathbb{N}$, then $|w(z)|\le |z|^m$ for all $z\in \mathbb{D}$.}"
 
\subsection{Proof of Theorem \ref{Thm1}} Let $g(z)=\sum_{n=0}^{\infty}b_n z^n\in S(f)$, where $f\in \widehat{C_0}(\alpha)$, $\alpha\in[1,2]$. Then the condition $g\prec f$ and the growth theorem \cite[Corollary 2.4]{B12} lead to the fact that
\begin{equation*}
|g(z)-g(0)|\le |f'(0)|f_{\alpha}(r)\le 2\alpha\, d(f(0),\partial f(\mathbb{D}))f_{\alpha}(r)
\end{equation*}
so that (because $f(0)=g(0)$) for $m_0\in \mathbb{N}$, we have
\begin{equation}\label{eqcsub1}
|g(w_0(z))|\le |f(0)|+ 2\alpha\, d(f(0),\partial f(\mathbb{D}))f_{\alpha}(r^{m_0}).
\end{equation}
By using Lemma \ref{lemma1}, we obtain the following inequality
\begin{equation}\label{eqcsub2}
	\sum_{n=  N}^{\infty} |b_n|r^n \le \sum_{n=  N}^{\infty} |a_n|r^n= |f'(0)|\sum_{n=  N}^{\infty} A_nr^n\le 2\alpha\, d(f(0),\partial f(\mathbb{D}))\sum_{n=  N}^{\infty} A_nr^n
\end{equation}
for $|z|=r\le 1/3$. Clearly from \eqref{eqcsub1} and \eqref{eqcsub2}, we deduce that
\begin{align*}
|g(w_0(z))|+\sum_{n=N}^{\infty} |b_n|r^n&\le |f(0)|+ 2\alpha\, d(f(0),\partial f(\mathbb{D}))\bigg(\sum_{n=  N}^{\infty} A_nr^n+f_{\alpha}(r^{m_0})\bigg)\\
&= |f(0)|+ 2\alpha \,d(f(0),\partial f(\mathbb{D}))\bigg(F_{\alpha,m_0}^{ N}(r)+\cfrac{1}{2\alpha}\bigg)\\&:=\Phi_{f,\alpha,m_0}^{ N}(r),
\end{align*}
where $F_{\alpha,m_0}^{ N}(r)$ is given by \eqref{eqcsub3}.
 Obviously, $F_{\alpha,m_0}^{ N}(r)$ is an increasing function of $r$ on $[0,1)$. Moreover, $F_{\alpha,m_0}^{ N}(0)<0$ and $\lim_{r\to 1}F_{\alpha,m_0}^{ N}(r)=+\infty$ and hence the equation $F_{\alpha,m_0}^{ N}(r)=0$ has the unique positive root $r_{\alpha,m_0}^{ N}$ in $(0,1)$. In order to complete the proof, it suffices to show that $\Phi_{f,\alpha,m_0}^{ N}(r)\le |f(0)|+  d(f(0),\partial f(\mathbb{D}))$ holds for $|z|=r\le \{r_{\alpha,m_0}^{ N},1/3\}$. 
 
 \medskip
If $r_{\alpha,m_0}^{ N}\le 1/3$, then we show that $r_{\alpha,m_0}^{ N}$ cannot be improved. For $\alpha
\in [1,2]$, consider $g=f=f_{\alpha}$. Then, for $z=r$, 
\begin{align*}
|f(w_0(r))|+	\sum_{n=  N}^{\infty} |a_n|r^n&=|f_{\alpha}(0)|+2\alpha \, d(f_{\alpha}(0),\partial f_{\alpha}(\mathbb{D}))\bigg(f_{\alpha}(r^{m_0})+\sum_{n=  N}^{\infty} A_nr^n\bigg)\\
&> d(f_{\alpha}(0),\partial f_{\alpha}(\mathbb{D}))=\cfrac{1}{2\alpha}
\end{align*}
holds for $r>r_{\alpha,m_0}^{ N}$ and hence the Bohr radius $r_{\alpha,m_0}^{ N}$ is sharp.
\hfill{$\Box$}
 
\subsection{Proof of Theorem \ref{Thm2}} Let $f\in C_0(\alpha)$ and $w_0\in \mathcal{B}_0$. Then it follows,  for $|z|=r<1$, from Lemma \ref{lemmac0} and the classical the Schwarz lemma that
\begin{align}
|f(w_0(z))|&\le |f(0)|+|f'(0)|\sum_{n=1}^{\infty}A_n|w_0(z)|^n\nonumber\\
&=|f(0)|+2\alpha\,  d(f(0),\partial f(\mathbb{D})) f_{\alpha}(r^{m_0}).\label{eqc0}
\end{align}
On the other hand, the Distortion theorem \cite[Corollary 2.3]{B12} leads to
\begin{align}
	|f'(w_1(z))||w_2(z)|&\le |f'(0)|\cfrac{(1+|w_1(z)|)^{\alpha-1}}{(1-|w_1(z)|)^{\alpha+1}}\,r^{m_2}\nonumber\\
	&\le 2\alpha  \, d(f(0),\partial f(\mathbb{D}))\cfrac{(1+r^{m_1})^{\alpha-1}}{(1-r^{m_1})^{\alpha+1}}r^{m_2}\label{eqc4}.
\end{align}
The last inequality follows because $u(x)=(1+x)^{\alpha-1}/(1-x)^{\alpha+1}$ is an increasing function on $[0,1]$. Also, from Lemma \ref{lemmac0}(ii) we can write
\begin{align}
	\sum_{n=  N}^{\infty}|a_n|| w^*_n(z)|&\le |f'(0)|\sum_{n=  N}^{\infty}A_n| w^*_n(z)|\nonumber\\
	&\le 2\alpha \,  d(f(0),\partial f(\mathbb{D})) \sum_{n=  N}^{\infty}A_n r^{b_n}.\label{eqc5}
\end{align}
Thus combining the equations \eqref{eqc0}, \eqref{eqc4}, and \eqref{eqc5} we obtain 
\begin{align}
|f(w_0(z))|+&|f'(w_1(z))||w_2(z)|+\sum_{n=  N}^{\infty}|a_n|| w^*_n(z)|\nonumber\\&\le |f(0)|+2\alpha \, d(f(0),\partial f(\mathbb{D})) \bigg(f_{\alpha}(r^{m_0})+ \sum_{n=  N}^{\infty}A_n r^{b_n}+\cfrac{(1+r^{m_1})^{\alpha-1}}{(1-r^{m_1})^{\alpha+1}}\,r^{m_2}\bigg)\nonumber\\
&= |f(0)|+2\alpha\, d(f(0),\partial f(\mathbb{D}))\bigg(K_{\alpha,m_0,m_1,m_2}^{ N}(r)+\cfrac{1}{2\alpha}\bigg)\label{eqc3},
\end{align}
where 
$$
K_{\alpha,m_0,m_1,m_2}^{ N}(r)=f_{\alpha}(r^{m_0})+ \sum_{n=  N}^{\infty}A_n r^{b_n}+\cfrac{(1+r^{m_1})^{\alpha-1}}{(1-r^{m_1})^{\alpha+1}}\,r^{m_2}-\cfrac{1}{2\alpha}.
$$
Note that the function $K_{\alpha,m_0,m_1,m_2}^{ N}(r)$ is a strictly increasing function of $r$ in $(0,1)$, $K_{\alpha,m_0,m_1,m_2}^{ N}(0)$ is negative, and $\lim_{r\to 1}K_{\alpha,m_0,m_1,m_2}^{ N}(r)=\infty$ and hence there exists the unique positive root say $r_{\alpha,m_0,m_1,m_2}^{ N}$ of the equation \eqref{eqC2} in $(0,1)$. Thus the last quantity of the inequality \eqref{eqc3} is less than or equal to $|f(0)|+ d(f(0),\partial f(\mathbb{D}))$ for $r\le r_{\alpha,m_0,m_1,m_2}^{ N}$ and this completes the first part of the proof.

\medskip
Next we show that the radius $r_{\alpha,m_0,m_1,m_2}^{ N}$ is sharp. Let $w_i(z)=z^{m_i}\,(i=0,1,2)$, $ w^*_n(z)=z^{h(n)}$, and $f=f_{\alpha}$. Simple computation shows that
$$
|f_{\alpha}(w_0(z))|+|f_{\alpha}'(w_1(z))||w_2(z)|+\sum_{n=  N}^{\infty}A_n| w^*_n(z)|=|f_{\alpha}(z^{m_0})|+|f_{\alpha}'(z^{m_1})|r^{m_2}+\sum_{n=  N}^{\infty}A_nr^{h(n)}.
$$
After substituting $z=r$ in the above equation, we obtain
\begin{align*}
|f_{\alpha}(w_0(r))|+|f_{\alpha}'(w_1(r))||w_2(r)|
+&\sum_{n=  N}^{\infty}A_n| w^*_n(r)|\\&=|f_{\alpha}(0)|+2\alpha \, d(f_{\alpha}(0),\partial f_{\alpha}(\mathbb{D}))\bigg(K_{\alpha,m_0,m_1,m_2}^{ N}(r)+\cfrac{1}{2\alpha}\bigg)\\
  &> d(f_{\alpha}(0),\partial f_{\alpha}(\mathbb{D}))=\cfrac{1}{2\alpha}
 \end{align*}
which holds if, and only if, $r> r_{\alpha,m_0,m_1,m_2}^{ N}$.
This completes the proof.\hfill{$\Box$}


\subsection{Proof of Theorem \ref{Thm4}} 
For a given $f\in\widehat{C_p}$, let $g(z)=\sum_{n=0}^{\infty}b_n z^n \in S(f)$. Let the Koebe transform of $f$,
$$
F(z)=\cfrac{f\bigg(\cfrac{z+a}{1+\overline{a}z}\bigg)-f(a)}{(1-|a|^2)f'(a)}
$$
for any $a \in \mathbb{D}\setminus\{p\}$. Since for some $t\in \mathbb{R}$, $e^{-it}F(ze^{it})\in C_{|\frac{p-a}{1-\overline{a}p}|}$. Therefore, from \cite{KJW03}, we have
\begin{equation}\label{eqcpd}
d(f(0),\partial f(\mathbb{D}))\ge (p/(1+p)^2)|f'(0)|.
\end{equation}
As $F(z)=(f(z)-f(0))/f'(0)\in {C_p}$, it follows that (see \cite{Jen62})          
\begin{equation}\label{eqan}
|a_n|\le |f'(0)|\cfrac{1-p^{2n}}{(1-p^2)p^{n-1}}, \mbox{ for all $n\ge 1$}.
\end{equation}
In view of $g\prec f$ and \eqref{eqan} we have the following inequality
\begin{equation*}
|g(z)-g(0)|\le |f(z)-f(0)|\le \sum_{n=1}^{\infty} |a_n|r^n\le k_p(r)=|f'(0)|\sum_{n=1}^{\infty}\cfrac{1-p^{2n}}{(1-p^2)p^{n-1}}\,r^n.
\end{equation*}
Equivalently, for $w_0\in \mathcal{B}_{m_0}$, we obtain
\begin{equation}\label{eqcp1}
	|g(w_0(z))|\le |f(0)|+|f'(0)|\sum_{n=1}^{\infty}\cfrac{1-p^{2n}}{(1-p^2)p^{n-1}}r^{m_0n}=|f(0)|+|f'(0)|k_p(r^{m_0}).
\end{equation}
Also, Lemma \ref{lemmacp} gives that
\begin{equation}\label{eqp1}
\sum_{n=  N}^{\infty} |b_n|r^n\le |f'(0)|\sum_{n=  N}^{\infty}\cfrac{1-p^{2n}}{(1-p^2)p^{n-1}}\,r^n.
\end{equation}
Thus, using \eqref{eqcp1} and \eqref{eqp1} together with inequality \eqref{eqcpd}, we obtain
\begin{align*}
|g(w_0(z))|+\sum_{n=  N}^{\infty} |b_n|r^n&\le  |f(0)|+\cfrac{(1+p)^2}{p}~d(f(0),\partial f(\mathbb{D})) \bigg(\sum_{n=  N}^{\infty}\cfrac{1-p^{2n}}{(1-p^2)p^{n-1}}r^n+k_p(r^{m_0})\bigg)\\
&=|f(0)|+\cfrac{(1+p)^2}{p}~d(f(0),\partial f(\mathbb{D}))\bigg(G_{p,m_0}^{ N}(r)+\cfrac{p}{(1+p)^2}\bigg)\\
&:=T_{f,p,m_0}^{ N}(r),
\end{align*}
where $G_{p,m_0}^{ N}(r)$ is given by \eqref{eqcp}. Note that $G_{p,m_0}^{ N}(0)<0$ and since
$$
\sum_{n=  N}^{\infty}\cfrac{1-p^{2n}}{(1-p^2)p^{n-1}}p^n=\cfrac{p(1-p^{2N})}{(1-p^2)}+\sum_{n=  N+1}^{\infty}\cfrac{p(1-p^{n})(1+p^n)}{(1-p^2)}\ge \sum_{n=  N+1}^{\infty}p(1+p^n)
$$
diverges to infinity, we find that $G_{p,m_0}^{ N}(p)>0 $ while $G_{p,m_0}^{ N}(r)$ is strictly increasing in $(0,1)$. Therefore we conclude that the equation \eqref{eqcp} has the unique positive root say $r_{p,m_0}^{ N}$, $0<r_{p,m_0}^{ N}<p$ in $(0,1)$. This shows that $T_{f,p,m_0}^{ N}(r)\le |f(0)|+d(f(0),\partial f(\mathbb{D}))$, for $r\le r_{p,m_0}^{ N}$.

\medskip
For the equality, we consider the function $w_0(z)=z^{m_0}$ and $f=g=	k_{p}$ given by \eqref{Extremalp}. For this function it is well known that $\widehat{\mathbb{C}}\setminus k_p(\mathbb{D})=[-p/(1-p)^2,-p/(1+p)^2]$ (see \cite[p.~137]{AW09}) and hence we obtain $d(k_p(0),\partial k_p(\mathbb{D}))=p/(1+p)^2$. Taking $z=r$, then the left side of the inequality \eqref{eqcpbohr} reduces to
\begin{align*}
	|k_p(r^{m_0})|+&\sum_{n=  N}^{\infty} \cfrac{1-p^{2n}}{(1-p^2)p^{n-1}}r^n\\&=|k_p(0)|+\cfrac{(1+p)^2}{p}~d(k_p(0),\partial k_p(\mathbb{D}))\bigg(k_p(r^{m_0})+\sum_{n=  N}^{\infty}\cfrac{1-p^{2n}}{(1-p^2)p^{n-1}}r^n\bigg)\\
	&>d(k_p(0),\partial k_p(\mathbb{D}))=\cfrac{p}{(	1+p)^2},
\end{align*}
which holds if, and only if, $r> r_{p,m_0}^{ N}$, which means that the number $r_{p,m_0}^{ N}$ cannot be improved. This shows the sharpness.
\hfill{$\Box$} 


\bigskip
\noindent
{\bf Acknowledgement.}
The first author thanks SERB-CRG and the second author thanks IIT Bhubaneswar for providing Institute Post Doctoral Fellowship.

\medskip
\noindent
{\bf Conflict of Interests.} The authors declare that there is no conflict of interests 
regarding the publication of this paper.

\end{document}